\documentclass[12pt,reqno]{amsart}

\usepackage{amsmath,amssymb,amsfonts,amsthm,latexsym,graphicx,multirow,hyperref,enumerate,tikz,lmodern}

\newcommand\A{\mathrm{A}} \newcommand\Alt{\mathrm{Alt}} \newcommand\Aut{\mathrm{Aut}}
 \newcommand\bfO{\mathbf{O}}
\newcommand\C{\mathrm{C}} \newcommand\Cos{\mathsf{Cos}}
\newcommand\Cay{\mathrm{Cay}} \newcommand\Cen{\mathbf{C}}
\newcommand\D{\mathrm{D}}

\newcommand\Inn{\mathrm{Inn}}
\newcommand\magma{{\sc Magma} }
\newcommand\Nor{\mathbf{N}}
\newcommand\Out{\mathrm{Out}}
\newcommand\Rad{\mathrm{Rad}}
 \newcommand\Soc{\mathrm{Soc}}  \newcommand\Sym{\mathrm{Sym}}

\newtheorem{theorem}{Theorem}[section]
\newtheorem{lemma}[theorem]{Lemma}
\newtheorem{proposition}[theorem]{Proposition}

\newtheorem{conjecture}[theorem]{Conjecture}
\theoremstyle{definition}

\newtheorem*{remark}{Remark}

\begin{document}

\title[Half-arc-transitive covers]{Constructing infinitely many half-arc-transitive covers of tetravalent graphs}

\author[Spiga]{Pablo Spiga}
\address{(Spiga) Dipartimento di Matematica e Applicazioni, University of Milano-Bicocca, Via Cozzi 55, 20125 Milano, Italy}
\email{pablo.spiga@unimib.it}

\author[Xia]{Binzhou Xia}
\address{(Xia) School of Mathematics and Statistics\\The University of Melbourne\\Parkville, VIC 3010\\Australia}
\email{binzhoux@unimelb.edu.au}


\maketitle

\begin{abstract}
We prove that, given a finite graph $\Sigma$ satisfying some mild conditions, there exist infinitely many tetravalent half-arc-transitive normal covers of $\Sigma$. Applying this result, we establish the existence of infinite families of finite tetravalent half-arc-transitive graphs with certain vertex stabilizers, and classify the vertex stabilizers up to order $2^8$ of finite connected tetravalent half-arc-transitive graphs. This sheds some new light on the longstanding problem of classifying the vertex stabilizers of finite tetravalent half-arc-transitive graphs.


\end{abstract}

\section{Introduction}

Let $\Gamma$ be a graph and let $G$ be a subgroup of the automorphism group $\Aut(\Gamma)$ of $\Gamma$. We say that $G$ is \emph{vertex-transitive}, \emph{edge-transitive} or \emph{arc-transitive} if $G$ acts transitively on the vertex set, edge set or the set of ordered pairs of adjacent vertices, respectively, of $\Gamma$. If $G$ is vertex-transitive and edge-transitive but not arc-transitive, then we say that $G$ is \emph{half-arc-transitive}. The graph $\Gamma$ is said to be half-arc-transitive if $\Aut(\Gamma)$ is half-arc-transitive.

Numerous papers have been published on half-arc-transitive graphs over the last half a century (see the survey papers~\cite{CPS2015,Marusic1998}), most of which are on those of valency $4$, the smallest valency of half-arc-transitive graphs. However, somewhat surprisingly, not so many examples of tetravalent half-arc-transitive graphs are known in the literature (see~\cite{census}), compared with the considerable attention they have received.

For a graph $\Gamma$ and a group $N$ such that $N$ is normal in $G$ for some vertex-transitive subgroup $G$ of $\Aut(\Gamma)$, the \emph{normal quotient} $\Gamma/N$ is the graph whose vertex set $V(\Gamma/N)$ is the set of $N$-orbits on the vertex set $V(\Gamma)$ of $\Gamma$, with an edge of $\Gamma/N$ between vertices $\Delta$ and $\Omega$ if and only if there is an edge of $\Gamma$ between $\alpha$ and $\beta$ for some $\alpha\in\Delta$ and $\beta\in\Omega$. Such a graph $\Gamma$ is called a \emph{normal cover} of the graph $\Gamma/N$. Broadly speaking, in this paper, given a graph $\Sigma$ satisfying some mild conditions, we establish the existence of infinitely many tetravalent half-arc-transitive graphs that are normal covers of $\Sigma$.

Let $p$ be a prime number. For a positive integer $m$, denote the largest power of $p$ dividing $m$ by $m_p$. Moreover, given a finite group $X$, let $\bfO_p(X)$ denote the largest normal $p$-subgroup of $X$. Our main result is as follows.

\begin{theorem}\label{ThmCover}
Let $\Sigma$ be a finite connected tetravalent graph and let $T$ be a nonabelian simple half-arc-transitive subgroup of $\Aut(\Sigma)$. Then, for each prime number $p$, such that $p>|T|_2$ and $p$ is coprime to $|T|$, there exists a finite connected tetravalent graph $\Gamma$ satisfying the following:
\begin{enumerate}[{\rm (a)}]
\item $\Gamma$ is half-arc-transitive;
\item $\Aut(\Gamma)$ has vertex stabilizer isomorphic to that of $T$;
\item $\bfO_p(\Aut(\Gamma))\ne 1$, $\Aut(\Gamma)/\bfO_p(\Aut(\Gamma))\cong T$ and $\Gamma/\bfO_p(\Aut(\Gamma))\cong\Sigma$.
\end{enumerate}
\end{theorem}


Although it is not hard to construct a graph $\Gamma$ with a half-arc-transitive group $G$ of automorphisms, it is in general not known  whether $\Aut(\Gamma)$ is larger than $G$ to possibly make $\Aut(\Gamma)$ arc-transitive on $\Gamma$. In this sense, the significance of Theorem~\ref{ThmCover} is asserting the existence (under some mild conditions) of infinitely many half-arc-transitive graphs which are normal covers of a given connected tetravalent graph, even if the given graph is not itself half-arc-transitive. Thus, with the help of Theorem~\ref{ThmCover}, one can construct infinitely many connected tetravalent half-arc-transitive graphs with some exotic  vertex stabilizers, and we will present some examples in this paper.

For a half-arc-transitive graph $\Gamma$, the vertex stabilizer in $\Aut(\Gamma)$ will be called the \emph{vertex stabilizer} of $\Gamma$. It is not hard to construct half-arc-transitive graphs with abelian vertex stabilizers (see for instance~\cite{Marusic2005}). However, half-arc-transitive graphs with nonabelian vertex stabilizers are much more elusive and  the problem of constructing half-arc-transitive graphs with nonabelian vertex stabilizers  has received extensive attention and considerable effort  (see for instance~\cite{CM2003,CPS2015,Spiga2016,Xia}). The first infinite family of half-arc-transitive graphs with nonabelian vertex stabilizers was only constructed very recently in~\cite{Xia}. The vertex stabilizers in~\cite{Xia} are isomorphic to $\D_8\times\D_8\times \C_2^{m-6}$ for integers $m$ with $m\geqslant7$.

In Example~\ref{ExInfiniteFamily} we construct a finite connected tetravalent graph $\Sigma_m$ for every integer $m\geqslant4$ such that $\Sigma_m$ admits a half-arc-transitive action of the alternating group $\A_{2^m}$ with vertex stabilizer $\D_8\times\C_2^{m-3}$. Then, by applying Theorem~\ref{ThmCover} to the graphs in Example~\ref{ExInfiniteFamily} and to the graphs in~\cite{Xia}, we obtain the following result:

\begin{theorem}\label{ThmInfiniteFamily}
For every integer $m\geqslant4$, there exist infinitely many finite connected tetravalent half-arc-transitive graphs with vertex stabilizer $\D_8\times\C_2^{m-3}$ and, for every integer $m\geqslant7$, there exist infinitely many finite connected tetravalent half-arc-transitive graphs with vertex stabilizer  $\D_8\times \D_8\times \C_2^{m-6}$.
\end{theorem}

A group $H=\langle a_1,\dots,a_m\rangle$ is said to be \emph{concentric} if $|\langle a_i,\dots,a_j\rangle|=2^{j-i+1}$ for all $1\leqslant i<j\leqslant m$ and there exists a group isomorphism
\[
\varphi:\langle a_1,\dots,a_{m-1}\rangle\to\langle a_2,\dots,a_m\rangle
\]
such that $a_i^\varphi=a_{i+1}$ for $i=1,\dots,m-1$. (Note in the definition that each $a_i$ is necessarily an involution if $m\geqslant3$.) The study of concentric groups dates back to Glauberman~\cite{Glauberman1969,Glauberman1971} about 50 years ago and was made systematic by Maru\v{s}i\v{c} and Nedela~\cite{MN2001} in 2001. It was proved in~\cite{MN2001} that a group $H$ is concentric if and only if there exist a connected tetravalent graph $\Gamma$ and a subgroup $G$ of $\Aut(\Gamma)$ such that $G$ is half-arc-transitive with vertex stabilizer $H$. Moreover, Maru\v{s}i\v{c} and Nedela gave a characterization of concentric groups in terms of their defining relations~\cite[Theorem~5.5]{MN2001} and determined the concentric groups of order up to $2^8$~\cite[Theorem~6.3]{MN2001}. Let
\begin{align*}
\mathcal{H}_7=\langle a_1,\dots,a_7\mid\ &a_i^2=1\text{ for }i\leqslant7,\ (a_ia_j)^2=1\text{ for }|i-j|\leqslant4,\\
&(a_1a_6)^2=a_3,\ (a_2a_7)^2=a_4,\ (a_1a_7)^2=a_5\rangle.
\end{align*}

\begin{theorem}[Glauberman-Maru\v{s}i\v{c}-Nedela]\label{ThmMN2001}
The following are precisely the concentric groups of order at most $2^8$:
\begin{align*}
\C_2^m\text{ for }1\leqslant m\leqslant8,&\quad\D_8\times\C_2^{m-3}\text{ for }3\leqslant m\leqslant8,\\
\D_8\times\D_8\times\C_2^{m-6}\text{ for }6\leqslant m\leqslant8,&\quad\mathcal{H}_7\times\C_2^{m-7}\text{ for }7\leqslant m\leqslant8.
\end{align*}
\end{theorem}

Maru\v{s}i\v{c}~\cite{Marusic2005} has shown that every nontrivial elementary abelian $2$-group is the vertex stabilizer of a connected tetravalent half-arc-transitive graph. Similar results have been proved for $\D_8$ by Conder and Maru\v{s}i\v{c}~\cite{CM2003} and for $\D_8\times\C_2$ by Conder, Poto\v{c}nik and \v{S}parl~\cite{CPS2015}. Moreover, the first author showed in~\cite{Spiga2016} that $\D_8\times\D_8$ and $\mathcal{H}_7$ are both vertex stabilizers of connected tetravalent half-arc-transitive graphs in a response to a problem posed in~\cite{MN2001}, and the second author recently proved in~\cite{Xia} that $\D_8\times\D_8\times\C_2^{m-6}$ is the vertex stabilizer of a connected tetravalent half-arc-transitive graph for every integer $m\geqslant7$. In light of these results and Theorem~\ref{ThmInfiniteFamily}, we see that the only concentric group of order at most $2^8$ that is not known to be the vertex stabilizer of a connected tetravalent half-arc-transitive graph is $\mathcal{H}_7\times\C_2$. In Example~\ref{ExSmallConcentric}, we apply Theorem~\ref{ThmCover} to construct connected tetravalent half-arc-transitive graphs with vertex stabilizer $\mathcal{H}_7\times\C_2$. This leads to the next theorem.

\begin{theorem}\label{ThmSmallConcentric}
Every concentric group of order at most $2^8$ is the vertex stabilizer of infinitely many finite connected tetravalent half-arc-transitive graphs.
\end{theorem}


We prove Theorem~\ref{ThmCover} in Section~\ref{sec:cover}. Then in Section~\ref{sec:1} we construct some connected tetravalent graphs admitting a half-arc-transitive nonabelian simple group action with vertex stabilizer $\mathcal{H}_7\times\C_2$ and $\D_8\times\C_2^{m-3}$ for $m\geqslant3$, respectively, which will be used in Section~\ref{sec:4} to  prove Theorems~\ref{ThmInfiniteFamily} and~\ref{ThmSmallConcentric}. In Section~\ref{sec:5} we briefly discuss the relevance of our work and a conjecture of D\v{z}ambi\'c-Jones and Conder concerning faithful amalgams. We also include a natural open problem at the end of Section~\ref{sec:5}.

\section{Proof of Theorem~\ref{ThmCover}}\label{sec:cover}

For a group $X$, let $\Soc(X)$ denote the socle of $X$ and let $\Rad(X)$ denote the maximal normal solvable subgroup of $X$. Let $\Gamma$ be a graph, let $G$ be a vertex-transitive subgroup of $\Aut(\Gamma)$ and let $N$ be a normal subgroup of $G$. Then the group $G$ induces a vertex-transitive subgroup of $\Aut(\Gamma/N)$. Denote by $\alpha^N$ and $\beta^N$ the $N$-orbits containing the vertices $\alpha$ and $\beta$, respectively, of $\Gamma$. If $\alpha^N$ and $\beta^N$ are adjacent in $\Gamma/N$, then each vertex in $\alpha^N$ is adjacent to the same number of vertices in $\beta^N$ (because $N$ is transitive on both sets). Moreover, the stabilizer in $G$ of the vertex $\alpha^N$ in $\Gamma/N$ is $G_\alpha N$.

See~\cite[Subsection~2.2]{PS2019} for the definition of \emph{regular covering projection}, \emph{lift} and \emph{group of covering transformations}.

\begin{proof}[Proof of Theorem~$\ref{ThmCover}$]
Let $\Sigma$ and $T$ be as in Theorem~\ref{ThmCover} and let $p$ be a prime number such that $p>|T|_2$ and $p$ is coprime to $|T|$. Viewing~\cite[Corollary~8]{PS2019} and applying~\cite[Theorem~6]{PS2019} with the prime $p$, the graph $\Sigma$ and the group of automorphisms $T$, we obtain a regular covering projection $\wp:\Gamma\to\Sigma$ such that the following hold:
\begin{enumerate}[{\rm (i)}]
\item $\Gamma$ is finite;
\item the maximal group that lifts along $\wp$ is $T$;
\item the group of covering transformations of $\wp$ is a $p$-group.
\end{enumerate}
Let $A=\mathrm{Aut}(\Gamma)$, let $G$ be the subgroup of $A$ that $T$ lifts to along $\wp$, and let $P$ be the group of covering transformations of $\wp$. Then conclusion~(iii) shows that $P$ is a $p$-group, and $G/P\cong T$ is nonabelian simple. Since $P$ is a normal solvable subgroup of $G$, it follows that $P=\Rad(G)$. Moreover, we deduce from conclusion~(ii) and~\cite[Lemma~1]{PS2019} that
\begin{equation}\label{Eqn1}
\Nor_A(P)=G.
\end{equation}
Since $P=\Rad(G)$ is characteristic in $G$, we derive that $P$ is normal in $\Nor_A(G)$, that is, $\Nor_A(G)\leqslant\Nor_A(P)$. Thus it follows from~\eqref{Eqn1} that
\begin{equation}\label{Eqn2}
\Nor_A(G)=G.
\end{equation}
We aim to prove that $A=G$, from which the proof of Theorem~\ref{ThmCover} immediately follows. Assume for a contradiction that $A>G$. Then $G<B$ for some subgroup $B$ of $A$ such that $G$ is maximal in $B$.

Let $\alpha$ be a vertex of $\Gamma$. Since $T$ is half-arc-transitive on $\Sigma$, the group $G$ is half-arc-transitive on $\Gamma$. This implies that $G_\alpha$ is a $2$-group and $B=GB_\alpha$ is edge-transitive and vertex-transitive on $\Gamma$. It follows that $|B:G|=|GB_\alpha:G|=|B_\alpha:G_\alpha|$ divides $|B_\alpha|$. As $B_\alpha$ is a $\{2,3\}$-group and $p>|T|_2\geqslant5$, we infer that $p$ is coprime to $|B:G|$. Since $p$ is coprime to $|T|=|G/P|$, we see that $P$ is a Sylow $p$-subgroup of $B$. According to Sylow's theorem, the number of Sylow $p$-subgroups of $B$ is $|B:\Nor_B(P)|\equiv1\pmod p$ and so $p$ divides $|B:\Nor_B(P)|-1$. By~\eqref{Eqn1} we have $\Nor_B(P)=G$. Hence
\begin{equation}\label{Eqn3}
p\mid(|B:G|-1).
\end{equation}

Let $K$ be the core of $G$ in $B$. Then $K\unlhd B$, $K\leqslant G$, and the action of $B/K$ on the set $\Omega$ of right cosets of $G/K$ in $B/K$ is faithful and primitive of degree $|B:G|$. Since both $K$ and $P$ are normal in $G$, the group $KP$ is normal in $G$, which implies that $KP/P$ is normal in $G/P$. As $G/P\cong T$ is a simple group, we deduce that either $G=KP$ or $K\leqslant P$.

\vspace{1em}
\noindent\textbf{Case 1.} $G=KP$.
\vspace{1em}

\noindent In this case, $P\cap K$ is a normal subgroup of $K$ with
\[
K/(P\cap K)\cong KP/P=G/P\cong T
\]
nonabelian simple. Since $P\cap K$ is solvable, we conclude that
\[
P\cap K=\Rad(K)
\]
is characteristic in $K$. As $K$ is normal in $B$, it follows that
\[
P\cap K\unlhd B.
\]
Note that $|G/K|=|KP/K|=|P/(P\cap K)|$ is a power of $p$ and $G\neq K$ by~\eqref{Eqn2}. We have
\[
|G/K|=p^n
\]
for some positive integer $n$.

Suppose that $|B_\alpha|$ is divisible by $3$. Then $B_\alpha$ is $2$-transitive on the neighborhood of $\alpha$ in $\Gamma$, and so it follows from a result of Gardiner (see for instance~\cite[Lemma~2.3]{FLX2004}) that $|B_\alpha|$ divides $2^43^6$. Now $B/K$ is a primitive group of degree $|B:G|=|B_\alpha:G_\alpha|$ dividing $2^33^6$ such that the point stabilizer $G/K$ is a $p$-group. We deduce from~\cite{LL2014} that $B/K$ is an affine group of degree $3^k$ with $3\leqslant k\leqslant6$, and $\Soc(B/K)$ is the unique Sylow $3$-subgroup of $B/K$. Since $B/K=(G/K)(B_\alpha K/K)$ and $|G/K|$ is coprime to $3$, it follows that $\Soc(B/K)\unlhd B_\alpha K/K\cong B_\alpha/K_\alpha$. Note that
\[
|\Soc(B/K)|=3^k=|B:G|=|B_\alpha:G_\alpha|=|B_\alpha|_3
\]
as $G_\alpha$ is a $2$-group. We conclude that the Sylow $3$-subgroup of $B_\alpha$ is elementary abelian of order $3^k\geqslant3^3$. The structure of the vertex stabilizer $B_\alpha$ is described in~\cite[Table~1]{Potocnik2009}, which shows that $B_\alpha$ cannot have an elementary abelian Sylow $3$-subgroup of order at least $3^3$, a contradiction. Thus $B_\alpha$ is a $2$-group, and so $|B:G|=|B_\alpha:G_\alpha|$ is a power of $2$, say,
\[
|B:G|=2^\ell.
\]
Note that $\ell>1$ by~\eqref{Eqn2}.

Since $|B:G|=2^\ell$ and $|G:K|=p^n$, we see that $B/K$ has order $2^\ell p^n$ and thus is solvable. Moreover, as $B/K$ is a primitive group of degree $2^\ell$, it follows that $\Soc(B/K)$ is an elementary abelian group of order $2^\ell$. Let $H$ be the subgroup of $B$ such that $H/K=\Soc(B/K)$. The reader may find Figure~\ref{Fig1} useful at this point.

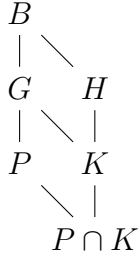
\begin{figure}[htbp]
\begin{tikzpicture}[node distance=1cm]
\node(A0){$B$};
\node(A1)[below of=A0]{$G$};
\node(A3)[right of=A1]{$H$};
\node(A6)[below of=A3]{$K$};
\node(A7)[below of=A1]{$P$};
\node(A8)[below of=A6]{$P\cap K$};
\draw(A0)--(A1);
\draw(A0)--(A3);
\draw(A1)--(A6);
\draw(A3)--(A6);
\draw(A1)--(A7);
\draw(A6)--(A8);
\draw(A7)--(A8);
\end{tikzpicture}
\caption{The structure of $B$}\label{Fig1}
\end{figure}

Let $\overline{B}=B/(P\cap K)$, $\overline{H}=H/(P\cap K)$, $\overline{K}=K/(P\cap K)$ and $\overline{C}=\Cen_{\overline{H}}(\overline{K})$. Then $\overline{K}\cong T$, and both $\overline{H}$ and $\overline{K}$ are normal in $\overline{B}$. It follows that $\overline{C}=\overline{H}\cap\Cen_{\overline{B}}(\overline{K})\unlhd\overline{B}$, and
\[
\overline{H}/\overline{C}\lesssim\Aut(\overline{K})\cong\Aut(T).
\]
Moreover,
\begin{equation}\label{Eqn5}
\overline{C}\,\overline{K}/\overline{C}\cong\overline{K}/(\overline{K}\cap\overline{C})\cong\Inn(\overline{K})\cong\Inn(T).
\end{equation}
Thus $\overline{H}/(\overline{C}\,\overline{K})\lesssim\Out(T)$. Let $C$ be the subgroup of $H$ containing $P\cap K$ such that $C/(P\cap K)=\overline{C}$. Then
\[
C\unlhd B
\]
and $H/(CK)\lesssim\Out(T)$. Now $CK\unlhd B$ and so $CK/K\unlhd B/K$. As $CK/K\leqslant H/K$ and $H/K=\Soc(B/K)$ is a minimal normal subgroup of the affine primitive group $B/K$, it follows that either $CK/K=1$ or $CK/K=H/K$. If $CK/K=1$, then the elementary abelian $2$-group $H/K=H/(CK)$ is isomorphic to a subgroup of $\Out(T)$, which implies that
\[
|B:G|=2^\ell=|H/K|\leqslant |\Out(T)|_2\leqslant |T|_2<p,
\]
contradicting~\eqref{Eqn3}. (Observe that the inequality $|\Out(T)|_2\leqslant |T|_2$ follows by inspecting the list of finite simple groups.) Therefore, $CK/K=H/K$ and hence $H=CK$. This in turn with~\eqref{Eqn5} implies that
\[
H/C=CK/C\cong\overline{C}\,\overline{K}/\overline{C}\cong T.
\]
Note that $T$ is the unique nonsolvable composition factor of $H$ as $H/K$ is solvable and $K$ is a $p$-group extended by $T$. We then conclude that
\[
C=\Rad(H).
\]
Consequently,
\[
C\cap K=\Rad(H)\cap K=\Rad(K)=P\cap K
\]
and so
\[
|C/(P\cap K)|=|C/(C\cap K)|=|CK/K|=|H/K|=2^\ell.
\]
The reader may find Figure~\ref{Fig2} useful at this point.

\begin{figure}[htbp]
\begin{tikzpicture}[node distance=2cm]
\node(A0){$B$};
\node(A1)[below of=A0]{$G$};
\node(A3)[right of=A1]{$H$};
\node(A6)[below of=A3]{$K$};
\node(A7)[below of=A1]{$P$};
\node(A8)[below of=A6]{$P\cap K$};
\node(A9)[right of=A6]{$C$};
\path[-] (A0) edge node [left]{$2^\ell$}(A1);
\path[-] (A0) edge node [right]{\hspace{0.2em}$p^n$}(A3);
\path[-] (A1) edge node [left]{$T$}(A7);
\path[-] (A7) edge node [right]{\hspace{0.2em}$p^n$}(A8);
\path[-] (A1) edge node [right]{\hspace{0.2em}$p^n$}(A6);
\path[-] (A3) edge node [right]{\hspace{-0.1em}$2^\ell$}(A6);
\path[-] (A6) edge node [right]{\hspace{-0.25em}$T$}(A8);
\path[-] (A9) edge node [right]{$T$}(A3);
\path[-] (A9) edge node [right]{$2^\ell$}(A8);
\end{tikzpicture}
\caption{More detailed structure of $B$} \label{Fig2}
\end{figure}
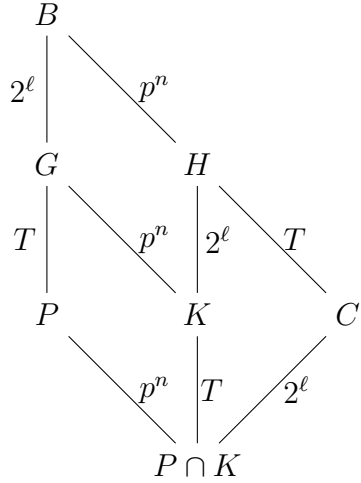


Consider the quotient graph $\Gamma/C$. Let $N$ be the kernel of $B$ acting on $V(\Gamma/C)$. Since $H$ is a normal subgroup of $B$ with index $p^n$ odd and $B_\alpha$ is a $2$-group, we have $B_\alpha\leqslant H$. Consequently, $N=CN_\alpha\leqslant CB_\alpha\leqslant H$. Moreover, $N=CN_\alpha$ is a $\{2,p\}$-group and thus is solvable. Hence $N\leqslant\Rad(H)=C$. This shows that the action of $B/C$ on $V(\Gamma/C)$ is faithful. Suppose that $C_\alpha\neq1$. Then the number of orbits of $C_\alpha$ on the neighborhood of $\alpha$ in $\Gamma$ is less than $4$. It follows that the valency of $\Gamma/C$ is less than $4$ and so must be $1$ or $2$, being a divisor of $4$. Thereby we conclude that $B/C\leqslant\Aut(\Gamma/C)$ is solvable, a contradiction. Thus $C_\alpha=1$.

As $C_\alpha=1$, the orbits of $C$ on $V(\Gamma)$ have size $|C|$. Since $C$ is normal in $B$ and $B$ is transitive on $V(\Gamma)$, it follows that $|C|$ divides $|V(\Gamma)|$. Hence $|C|$ divides $|G|$ as $G$ is transitive on $V(\Gamma)$. In particular, $|C|_2\leqslant|G|_2$. As $|C|_2=|C/(P\cap K)|_2=2^\ell$ and $|G|_2=|G/P|_2=|T|_2$, we then obtain $2^\ell\leqslant|T|_2$. This together with~\eqref{Eqn3} implies that $p<|B:G|=2^\ell\leqslant|T|_2$, contradicting our choice of $p$.

\vspace{1em}
\noindent\textbf{Case 2.} $K\leqslant P$.
\vspace{1em}

\noindent Let $\overline{B}=B/K$, $\overline{G}=G/K$, $\overline{P}=P/K$ and $\overline{H}=H/K=\Soc(\overline{B})$. Recall that $\overline{B}$ acts primitively and faithfully on the set of right cosets of $\overline{G}$ in $\overline{B}$, and
\[
|\overline{B}:\overline{G}|=|B:G|=|GB_\alpha:G|=|B_\alpha:G_\alpha|.
\]
As $B_\alpha$ is a $\{2,3\}$-group, we obtain $|\overline{B}:\overline{G}|=2^\ell3^k$ for some nonnegative integers $\ell$ and $k$. If $|B_\alpha|$ is divisible by $3$, then $B_\alpha$ is $2$-transitive on the neighborhood of $\alpha$ in $\Gamma$ and so~\cite[Lemma~2.3]{FLX2004} shows that $|B_\alpha|$ divides $2^43^6$. Consequently, either $\ell\leqslant3$ and $1\leqslant k\leqslant6$, or $k=0$.

Since $K$ is normal in $B$, we deduce from~\eqref{Eqn1} that $K\neq P$. Hence $K<P$ and so $\overline{P}$ is a nontrivial $p$-group. This shows that $\overline{G}$ is a nontrivial $p$-group extended by the nonabelian simple group $G/P\cong T$. Then as $\overline{G}$ is a point stabilizer of the primitive group $\overline{B}$ of degree $|\overline{B}:\overline{G}|=2^\ell3^k$, it follows from~\cite{LL2014} that $k=0$ and $\overline{B}$ is an affine primitive group of degree $2^\ell$. Hence $|\overline{H}|=2^\ell$, and so $H$ is a $\{2,p\}$-group.

Let $R=PH$. Then $R$ is a $\{2,p\}$-group and thus is solvable. Moreover, $R\unlhd B$, and as $P\leqslant G$, we have $B=HG=HPG=RG$. Hence
\[
B/R=RG/R\cong G/(G\cap R)\cong(G/P)/((G\cap R)/P).
\]
Since $G/P\cong T$ is simple, it follows that either $B/R=1$ or $B/R\cong T$. Clearly, $B\neq R$ as $R$ is solvable and $B$ is nonsolvable. Thus $B/R\cong T$ is nonabelian simple, which implies
\[
R=\Rad(B).
\]

Consider the quotient graph $\Gamma/R$. Let $M$ be the kernel of $B$ acting on $V(\Gamma/R)$. Then $M=RM_\alpha$. Since $M_\alpha\leqslant B_\alpha$ is a $2$-group, we see that $M$ is a $\{2,p\}$-group as $R$ is a $\{2,p\}$-group. Accordingly, $M$ is solvable, and so $M\leqslant\Rad(B)=R$. This shows that the action of $B/R$ on $V(\Gamma/R)$ is faithful. Suppose that $R_\alpha\neq1$. Then the number of orbits of $R_\alpha$ on the neighborhood of $\alpha$ in $\Gamma$ is less than $4$. It follows that the valency of $\Gamma/R$ is less than $4$ and so must be $1$ or $2$ as it divides $4$. Thereby we conclude that $B/R\leqslant\Aut(\Gamma/R)$ is solvable, a contradiction. Thus $R_\alpha=1$.

As $R_\alpha=1$, the orbits of $R$ on $V(\Gamma)$ have size $|R|$. Since $R$ is normal in $B$ and $B$ is transitive on $V(\Gamma)$, it follows that $|R|$ divides $|V(\Gamma)|$. Hence $|R|$ divides $|G|$ as $G$ is transitive on $V(\Gamma)$. In particular, $|R|_2\leqslant|G|_2$. As $|R|_2=|PH|_2=|H|_2=|H/K|_2=2^\ell$ and $|G|_2=|G/P|_2=|T|_2$, we then obtain $2^\ell\leqslant|T|_2$. This in conjunction with~\eqref{Eqn3} implies that $p<|G:B|=|\overline{B}:\overline{G}|=2^\ell\leqslant|T|_2$, contradicting our choice of $p$.
\end{proof}

\section{Examples}\label{sec:1}

Recall the standard construction of the \emph{coset graph} $\Cos(X,Y,S)$ for a group $X$ with a subgroup $Y$ and an inverse-closed subset $S$ of $X\setminus Y$ such that $S$ is finite union of double cosets of $Y$ in $X$. Such a graph has vertex set $[X:Y]$, the set of right cosets of $Y$ in $X$, and edge set $\{\{Yt,Yst\}\mid t\in X,\ s\in S\}$. It is easy to see that $\Cos(X,Y,S)$ has valency $|S|/|Y|$, and $X$ acts by right multiplication on $[X:Y]$ as a group of automorphisms of $\Cos(X,Y,S)$. Moreover, $\Cos(X,Y,S)$ is connected if and only if $X=\langle Y,S\rangle$.

\subsection{Example $\D_8$}\label{ExD8}
Let $G=\A_{10}$ and
\[
H=\langle(1, 2, 3, 4)(5, 6, 7, 8),(1, 4)(2, 3)(5, 7)(9, 10)\rangle<G.
\]
Clearly, $H\cong\D_8$. Let
\[
s=(1, 8, 10)(2, 7, 4, 6, 9, 3, 5)\in G.
\]
It can be checked immediately by the computational algebra system \magma~\cite{BCP1997} that
\[
\langle H,s\rangle=G,\quad |H:s^{-1}Hs|=2\quad\text{and}\quad s^{-1}\notin HsH.
\]
Then letting
\begin{equation}\label{GraphSigma}
\Sigma=\Cos(G,H,H\{s,s^{-1}\}H),
\end{equation}
we see that
\begin{itemize}
\item $\Sigma$ is a connected tetravalent graph;
\item $G$ acts faithfully and half-arc-transitively on $\Sigma$;
\item the vertex stabilizer in $G$ is $H\cong\D_8$.
\end{itemize}

\subsection{Example $\mathcal{H}_7\times\C_2$}\label{ExSmallConcentric}
Let
\begin{align*}
H=\langle a_1,\dots,a_8\mid\ &a_i^2=1\text{ for }i\leqslant8,\ (a_ia_j)^2=1\text{ for }|i-j|\leqslant5,\\
&(a_1a_7)^2=a_3,\ (a_2a_8)^2=a_4,\ (a_1a_8)^2=a_6\rangle.
\end{align*}
Then $H=\langle a_1,a_2,a_3,a_4,a_6,a_7,a_8\rangle\times\langle a_5\rangle\cong\mathcal{H}_7\times\C_2$. Let
\[
B=\langle a_1,\dots,a_7\rangle,\quad C=\langle a_2,\dots,a_8\rangle
\]
and let $\varphi:B\to C$ be the group isomorphism defined by
\[
a_i^\varphi=a_{i+1}\quad\text{for }\ i=1,\dots,7.
\]
Then $H=B\cup a_8B=C\cup a_1a_2C$. Let $x$ be the permutation on $H$ defined by
\[
b^x=b^\varphi\quad\text{and}\quad(a_8b)^x=a_1a_2b^\varphi\quad\text{for }\ b\in B.
\]
Denote the right regular representation of $H$ by $R:H\to\Sym(H)$. It can be checked easily by the computational algebra system \magma~\cite{BCP1997} that
\[
\langle R(H),x\rangle=\Alt(H),\quad x^{-1}R(H)x=R(C)\quad\text{and}\quad x^{-1}\notin R(H)xR(H).
\]
Then letting
\begin{equation}\label{GraphPi}
\Pi=\Cos(\Alt(H),R(H),R(H)\{x,x^{-1}\}R(H)),
\end{equation}
we see that
\begin{itemize}
\item $\Pi$ is a connected tetravalent graph;
\item $\Alt(H)$ acts faithfully and half-arc-transitively on $\Pi$;
\item the vertex stabilizer in $\Alt(H)$ is $R(H)\cong H\cong\mathcal{H}_7\times\C_2$.
\end{itemize}

\subsection{Example $\D_8\times\C_2^{m-3}$}\label{ExInfiniteFamily}
Let $m\geqslant4$ be an integer,
\[
H=\langle a,b\mid a^4=b^2=(ab)^2=1\rangle\times\langle c_1\rangle\times\dots\times\langle c_{m-3}\rangle,
\]
where $c_1,\dots,c_{m-3}$ are involutions. Clearly, $H\cong\D_8\times\C_2^{m-3}$. Let $h=a\prod_{i=0}^{\lceil(m-5)/2\rceil}c_{2i+1}$ and
\[
K=\langle a^2,b,c_1,\dots,c_{m-3}\rangle
=\langle a^2\rangle\times\langle b\rangle\times\langle c_1\rangle\times\dots\times\langle c_{m-3}\rangle.
\]
Then $K\cong\C_2^{m-1}$ and $H=K\cup aK=K\cup hK$. For convenience, put $c_i=1$ for $i\leqslant0$. Define $x\in\Aut(H)$ by letting
\[
a^x=a^{-1},\quad b^x=ab,\quad c_{2i+1}^x=c_{2i+1}\quad\text{and}\quad c_{2i+2}^x=a^2c_{2i+1}c_{2i+2}
\]
for $0\leqslant i\leqslant\lfloor(m-5)/2\rfloor$ and letting $c_{m-3}^x=a^2c_{m-3}$ in addition if $m$ is even. Define $\tau\in\Aut(K)$ by letting
\[
(a^2)^\tau=b,\quad b^\tau=a^2,\quad c_{2i+1}^\tau=c_{2i-1}c_{2i}c_{2i+2}\quad\text{and}\quad c_{2i+2}^\tau=c_{2i-1}c_{2i}c_{2i+1}
\]
for $0\leqslant i\leqslant\lfloor(m-5)/2\rfloor$ and letting $c_{m-3}^\tau=c_{m-3}$ in addition if $m$ is even.

Note that $x$ and $\tau$ are automorphisms of $H$ and $K$ respectively as the images of generators under $x$ and $\tau$ are generators of $H$ and $K$ satisfying the defining relations. Let $y$ be the permutation of $H$ such that $g^y=g^\tau$ and
\[
(hg)^y=
\begin{cases}
hg^\tau\quad&\text{if $m$ is odd,}\\
hg^\tau c_{m-3}\quad&\text{if $m$ is even},
\end{cases}
\]
for $g\in K$. Denote the right regular representation of $H$ by $R:H\to\Sym(H)$. It follows from~\cite[Lemmas~2.1~and~2.3]{CXZ2018} that $x$ and $y$ are both involutions and $\langle x,y,R(H)\rangle\leqslant\Alt(H)$. Let
\begin{equation}\label{GraphSigmam}
\Sigma_m=\Cos(\Alt(H),R(H),R(H)\{xy,yx\}R(H)).
\end{equation}

Fix the notation of $H$, $R$, $x$ and $y$ in this subsection, and let
\[
z=
\begin{cases}
R(h)yR(h^{-1})\quad&\text{if $m$ is odd,}\\
R(h)yR(h^{-1}c_{m-3})\quad&\text{if $m$ is even.}
\end{cases}
\]
According to~\cite[Lemma~2.1]{CXZ2018}, the permutation $z$ is an involution. Since $z\in\langle x,y,R(H)\rangle$, it follows from $\langle x,y,R(H)\rangle\leqslant\Alt(H)$ that $z\in\Alt(H)$. As $x$, $y$ and $z$ all fix $1\in H$, we may also view them as elements of $\Alt(H\setminus\{1\})$ when they cause no confusion. Use $\sqcup$ to denote a disjoint union of sets.

\begin{lemma}\label{LemDoubleCoset}
The following hold:
\begin{enumerate}[{\rm (a)}]
\item $R(H)xyR(H)=R(H)xy\sqcup R(H)xz$;
\item $R(H)yxR(H)=R(H)yx\sqcup R(H)zx$;
\item $R(H)\{xy,yx\}R(H)=R(H)xy\sqcup R(H)yx\sqcup R(H)xz\sqcup R(H)zx$.
\end{enumerate}
\end{lemma}

\begin{proof}
Note that $x$, $y$ and $z$ are all involutions. It is straightforward to verify that $x$ and $y$ normalize $R(H)$ and $R(K)$, respectively. If $yR(H)x\cap R(H)=R(H)$, then
\[
\langle x,y,R(H)\rangle\leqslant\Nor_{\Alt(H)}(R(H))<\Alt(H),
\]
contrary to~\cite[Lemmas~3.6~and~3.11]{CXZ2018}. Thus
\[
yxR(H)xy\cap R(H)=yR(H)y\cap R(H)\neq R(H).
\]
Since
\begin{align*}
yxR(H)xy\cap R(H)&=yR(H)y\cap R(H)\geqslant yR(K)y\cap R(K)=R(K)
\end{align*}
and $R(K)$ has index $2$ in $R(H)$, we then deduce that $yxR(H)xy\cap R(H)=R(K)$. In particular, $yxR(H)xy$ has index $2$ in $R(H)$, whence
\[
\frac{|R(H)xyR(H)|}{|R(H)|}=\frac{|R(H)|}{|yxR(H)xy\cap R(H)|}=2.
\]
Consequently,
\begin{equation}\label{Eqn6}
|R(H)yxR(H)|=|(R(H)yxR(H))^{-1}|=|R(H)xyR(H)|=2|R(H)|
\end{equation}
and thus
\begin{equation}\label{Eqn7}
|R(H)\{xy,yx\}R(H)|\leqslant|R(H)xyR(H)|+|R(H)yxR(H)|=4|R(H)|.
\end{equation}
Note from the definition of $z$ that
\[
xz\in xR(H)yR(H)=R(H)xyR(H).
\]
Hence $R(H)xz\subseteq R(H)xyR(H)$ and $R(H)zx\subseteq R(H)yxR(H)$. It is direct to verify that
\[
(a^2)^{xy}=b,\quad(a^2)^{yx}=ab,\quad(a^2)^{xz}=a^2b,\quad(a^2)^{zx}=a^3b,
\]
which shows that $xy$, $yx$, $xz$ and $zx$ are pairwise distinct. Then as $xy,yx,xz,zx\in\Alt(H)_1$ and $\Alt(H)_1$ forms a right transversal of $R(H)$ in $\Alt(H)$, it follows that $R(H)xy$, $R(H)yx$, $R(H)xz$ and $R(H)zx$ are pairwise disjoint. Therefore,
\[
R(H)xyR(H)\supseteq R(H)xy\sqcup R(H)xz,
\]
\[
R(H)yxR(H)\supseteq R(H)yx\sqcup R(H)zx
\]
and
\[
R(H)\{xy,yx\}R(H)\supseteq R(H)xy\sqcup R(H)yx\sqcup R(H)xz\sqcup R(H)zx.
\]
This combined with~\eqref{Eqn6} and~\eqref{Eqn7} yields the lemma.
\end{proof}

\begin{proposition}\label{PrpInfiniteFamily}
Let $m\geqslant4$ be an integer and let $\Sigma_m$ be the graph defined in~$\eqref{GraphSigmam}$. Then $\Sigma_m$ is a connected tetravalent graph admitting a half-arc-transitive action of $\A_{2^m}$ with vertex stabilizer $\D_8\times\C_2^{m-3}$.
\end{proposition}

\begin{proof}
Let $S=\{xy,yx,xz,zx\}\subset\Alt(H\setminus\{1\})$. According to~\cite[Lemmas~4.1~and~4.2]{CXZ2018}, $\Cay(\Alt(H\setminus\{1\}),\{x,y,z\})$ is connected, which means that $\Alt(H\setminus\{1\})=\langle x,y,z\rangle$. Consider the subgroup $W$ of even words of the generators $x$, $y$ and $z$ in $\Alt(H\setminus\{1\})$. Then $W$ has index $1$ or $2$ in $\Alt(H\setminus\{1\})$. Since $\Alt(H\setminus\{1\})$ is simple, it follows that $W=\Alt(H\setminus\{1\})$. Moreover, as $x$, $y$ and $z$ are involutions, we have
\[
W=\langle xy,xz,yz\rangle=\langle xy,xz,(xy)^{-1}(xz)\rangle=\langle xy,xz\rangle.
\]
Thus $\Alt(H\setminus\{1\})=\langle xy,xz\rangle$, and so $\Cay(\Alt(H\setminus\{1\}),S)$ is connected.

Let $\varphi\colon g\mapsto R(H)g$ be the mapping from $\Alt(H\setminus\{1\})$ to the vertex set of $\Sigma_m$. Since $\Alt(H\setminus\{1\})$ forms a right transversal of $R(H)$ in $\Alt(H)$, $\varphi$ is bijective. Moreover, for any $u$ and $v$ in $\Alt(H\setminus\{1\})$, $u$ is adjacent to $v$ in $\Cay(\Alt(H\setminus\{1\}),S)$ if and only if
\[
vu^{-1}\in S=\{xy,yx,xz,zx\},
\]
which is equivalent to
\[
R(H)vu^{-1}\in\{R(H)xy,R(H)yx,R(H)xz,R(H)zx\}.
\]
By Lemma~\ref{LemDoubleCoset}, this means that $u$ and $v$ are adjacent in $\Cay(\Alt(H\setminus\{1\}),S)$ if and only if
\[
R(H)vu^{-1}\subseteq R(H)S=R(H)\{xy,yx\}R(H),
\]
or equivalently, $R(H)u$ is adjacent to $R(H)v$ in $\Sigma_m$. Therefore, $\varphi$ is a graph isomorphism from $\Cay(\Alt(H\setminus\{1\}),S)$ to $\Sigma_m$. As a consequence, $\Sigma_m$ is a connected tetravalent graph.


Finally, Lemma~\ref{LemDoubleCoset} implies that $\{R(H)xy,R(H)xz\}$ and $\{R(H)yx,R(H)zx\}$ are the two orbits of the group $R(H)$ acting on the neighborhood of the vertex $R(H)$ in $\Sigma_m$. Thereby we conclude that the right multiplication action of $\Alt(H)$ on $\Sigma_m$ is half-arc-transitive. Since $R(H)\cong H\cong\D_8\times\C_2^{m-3}$ and $\Alt(H)\cong\A_{|H|}=\A_{2^m}$, this completes the proof.
\end{proof}

\section{Proofs of Theorem~\ref{ThmInfiniteFamily} and Theorem~\ref{ThmSmallConcentric}}\label{sec:4}

\begin{proof}[Proof of Theorem~$\ref{ThmInfiniteFamily}$]
Let $m\geqslant4$ be an integer and let $\Sigma_m$ be the graph defined by~\eqref{GraphSigmam}. Then as Proposition~\ref{PrpInfiniteFamily} asserts, $\Sigma_m$ is a connected tetravalent graph admitting a half-arc-transitive action of $\A_{2^m}$ with vertex stabilizer $\D_8\times\C_2^{m-3}$. Thus, by Theorem~\ref{ThmCover}, there exist infinitely many finite connected tetravalent half-arc-transitive graphs with vertex stabilizer $\D_8\times\C_2^{m-3}$.

Let $m\geqslant7$ be an integer and let $\Gamma_m$ be the graph defined in~\cite[Section~$3$]{Xia}. Then~\cite[Theorem~$1.2$]{Xia} asserts that  $\Gamma_m$ is a connected tetravalent half-arc-transitive graph whose automorphism group is isomorphic to $\A_{2^{m}}$ with vertex stabilizer $\D_8\times \D_8\times\C_2^{m-6}$. Thus, by Theorem~\ref{ThmCover}, there exist infinitely many finite connected tetravalent half-arc-transitive graphs with vertex stabilizer $\D_8\times\D_8\times\C_2^{m-6}$.
\end{proof}

\begin{proof}[Proof of Theorem~$\ref{ThmSmallConcentric}$]
According to~\cite[Theorem~1.1]{Marusic2005} and Theorem~\ref{ThmInfiniteFamily}, each of the groups
\[
\C_2^m\text{ with }1\leqslant m\leqslant8,\quad\D_8\times\C_2^{m-3}\text{ with }4\leqslant m\leqslant8,\quad\D_8^2\times\C_2^{m-6}\text{ with }7\leqslant m\leqslant8
\]
is the vertex stabilizer of infinitely many finite connected tetravalent half-arc-transitive graphs. By Example~\ref{ExD8},~\cite{CM2003},~\cite{Spiga2016} and Example~\ref{ExSmallConcentric}, each of the groups
\[
\D_8,\quad\D_8\times\D_8,\quad\mathcal{H}_7,\quad\mathcal{H}_7\times\C_2
\]
is the vertex stabilizer of a half-arc-transitive nonabelian simple group acting on a finite connected tetravalent graph. This implies that each of these groups is the vertex stabilizer of infinitely many finite connected tetravalent half-arc-transitive graphs by Theorem~\ref{ThmCover}. Hence every group in the list of Theorem~\ref{ThmMN2001} is the vertex stabilizer of infinitely many finite connected tetravalent half-arc-transitive graphs, and so Theorem~\ref{ThmSmallConcentric} is true.
\end{proof}

\section{Concluding remarks}\label{sec:5}

The work in this paper was inspired by some new ideas developed in~\cite{Xia}. We observe that there are many papers that recently keep the interest on half-arc-transitive graphs very high. For instance, in~\cite{Zhou2016}, the author made significant progress in the study of tetravalent non-normal half-arc-transitive Cayley graphs of prime power order, and answered two very important problems related to this topic. Similarly, in~\cite{ZZ2017}, the authors answered a long-standing problem regarding the existence of half-arc-transitive graphs of order twice a prime square.

The reader may have noticed that our key ingredient in this paper is~\cite{PS2019}. The main results of~\cite{PS2019} are rather general and apply to most actions of groups on graphs. Our application of~\cite{PS2019} in our work is rather successful, in our opinion, as we consider normal covers of graphs
\begin{center}
$(\dag)\qquad$  admitting a \textit{nonabelian simple} group of automorphisms.
\end{center}
Under this extra hypothesis, the results in~\cite{PS2019} can be combined with rather strong group-theoretic results based on CFSG and, as a consequence, we are able to obtain infinite families of graphs having exotic vertex stabilizers. As far as we are aware, this type of constructions is a novelty.

In light of the following conjecture originating from D\v{z}ambi\'c and Jones~\cite{DJ2013} and supported by Conder (see~\cite[Section~2]{Conder2019}), the hypothesis~$(\dag)$ does not seem strong. This suggests that Theorem~\ref{ThmCover} could be applied to show the existence of tetravalent half-arc-transitive graphs with other vertex stabilizers as well, and thus sheds light on classifying the vertex stabilizers of finite tetravalent half-arc-transitive graphs.

\begin{conjecture}[Conder-D\v{z}ambi\'c-Jones]
If $A$ and $B$ are finite groups and $C$ is a subgroup of $A\cap B$ of index at least $2$ in $A$ and at least $3$ in $B$, then all but finitely many alternating groups are homomorphic images of the amalgamated free product $A \ast_C B$.
\end{conjecture}

\begin{remark}
In fact, Marston Conder has a stronger conjecture:
\end{remark}

\begin{conjecture}[\cite{Conder2019}]
Let $A$ and $B$ be finite groups, let $C$ be a subgroup of $A\cap B$ of index at least $2$ in $A$ and at least $3$ in $B$, and let $K$ be the core of $C$ in the amalgamated free product $A \ast_C B$. Then all but finitely many alternating groups occur as the image of $A \ast_C B$ under some homomorphism that takes $A$ and $B$ to subgroups (of the alternating group) isomorphic to $A/K$ and $B/K$ respectively.
\end{conjecture}

We conclude this section by giving a natural generalization of the example in Subsection~\ref{ExSmallConcentric}. Let $m\geqslant7$ be an integer and let
\begin{align*}
H=\langle a_1,\dots,a_m\mid\ &a_i^2=1\text{ for }i\leqslant m,\ (a_ia_j)^2=1\text{ for }|i-j|\leqslant m-3,\\
&(a_1a_{m-1})^2=a_3,\ (a_2a_m)^2=a_4,\ (a_1a_m)^2=a_{m-2}\rangle.
\end{align*}
Then $H=\langle a_1,a_2,a_3,a_4,a_{m-2},a_{m-1},a_m\rangle\times\langle a_5,\dots,a_{m-3}\rangle\cong\mathcal{H}_7\times\C_2^{m-7}$. Let
\[
B=\langle a_1,\dots,a_{m-1}\rangle,\quad C=\langle a_2,\dots,a_m\rangle
\]
and let $\varphi:B\to C$ be the group isomorphism defined by
\[
a_i^\varphi=a_{i+1}\quad\text{for }\ i=1,\dots,m-1.
\]
Then $H=B\cup a_m B=C\cup a_1a_2C$. Let $x$ be the permutation on $H$ defined by
\[
b^x=b^\varphi\quad\text{and}\quad(a_mb)^x=a_1a_2b^\varphi\quad\text{for }\ b\in B.
\]
Denote the right regular representation of $H$ by $R$. Inspired by~\cite{Spiga2016} and the results in Subsection~\ref{ExSmallConcentric}, we make the following conjecture.

\begin{conjecture}
Let $H$, $R$ and $x$ be as above. Then
\[
\Cos(\Alt(H),R(H),R(H)\{x,x^{-1}\}R(H))
\]
is a connected tetravalent graph on which the right multiplication action of $\Alt(H)$ is half-arc-transitive.
\end{conjecture}

If this conjecture is true then Theorem~\ref{ThmCover} will imply that for every integer $m\geqslant7$ there exist infinitely many finite connected tetravalent half-arc-transitive graphs with vertex stabilizer $\mathcal{H}_7\times\C_2^{m-7}$.

\end{document}